\documentclass[10pt]{amsart}
\usepackage{amssymb}

\newtheorem{thm}{Theorem}[section]
\newtheorem{lemma}[thm]{Lemma}
\newtheorem{proposition}[thm]{Proposition}

\newtheorem{definition}[thm]{Definition}
\newtheorem{corollary}[thm]{Corollary}

\newcommand{\p}{\mathbb{P}}
\newcommand{\q}{\mathbb{Q}}

\newcommand{\dom}{\mathrm{dom}}

\newcommand{\add}{\textsc{Add}}

\begin{document}

\title{Quotients of strongly proper forcings and guessing models}

\author{Sean Cox and John Krueger}

\address{Sean Cox \\ Department of Mathematics and Applied Mathematics \\
Virginia Commonwealth University \\ 
1015 Floyd Avenue \\ 
PO Box 842014 \\ 
Richmond, Virginia 23284}
\email{scox9@vcu.edu}

\address{John Krueger \\ Department of Mathematics \\ 
University of North Texas \\
1155 Union Circle \#311430 \\
Denton, TX 76203}
\email{jkrueger@unt.edu}

\date{Submitted June 2014; revised June 2015}

\thanks{2010 \emph{Mathematics Subject Classification:} 
Primary 03E40; Secondary 03E35.}

\thanks{\emph{Key words and phrases.} strongly proper, 
approximation property, adequate 
set forcing, guessing model.}

\begin{abstract}
We prove that a wide class of strongly proper forcing posets have 
quotients with strong properties. 
Specifically, we prove that quotients of forcing posets which have universal 
strongly generic conditions on a stationary set of models by certain nice 
regular suborders satisfy the 
$\omega_1$-approximation property. 
We prove that the existence of stationarily many 
$\omega_1$-guessing models 
in $P_{\omega_2}(H(\theta))$, for sufficiently large cardinals 
$\theta$, is consistent with the continuum being arbitrarily large, solving a problem 
of Viale and Weiss \cite{weiss1}.
\end{abstract}

\maketitle

Many consistency results in set theory involve factoring a forcing poset $\q$ 
over a regular suborder $\p$ in a forcing extension by $\p$, 
and applying properties of the quotient forcing $\q / \dot G_\p$. 
We will be interested in the situation where $\q$ has strongly generic 
conditions for elementary substructures, 
and we wish the quotient $\q / \dot G_\p$ to have similar properties. 
For example, the quotient $\q / \dot G_\p$ having the 
approximation property is useful for constructing 
models in which there is a failure of square principles 
or related properties.

We introduce some variations of strongly generic conditions, including 
simple and universal conditions. 
Our main theorem regarding quotients is that if $\q$ is a forcing poset 
with greatest lower bounds 
for which there are stationarily many 
countable elementary substructures which have universal strongly generic 
conditions, and $\p$ is a regular suborder of $\q$ which relates in a 
nice way to $\q$, 
then $\p$ forces that $\q / \dot G_\p$ has the $\omega_1$-approximation property. 
Several variations of this theorem are given, as well as an example 
which shows that not all quotients of strongly proper forcings are well behaved. 

Previously Weiss introduced combinatorial principles which characterize 
supercompactness yet also make sense for successor cardinals 
(\cite{weiss1}, \cite{weiss2}). 
Of particular interest to us is the principle $\mathsf{ISP}(\omega_2)$, 
which asserts the existence of stationarily many $\omega_1$-guessing 
models in $P_{\omega_1}(H(\theta))$, for sufficiently large 
regular cardinals $\theta$. 
This principle follows from PFA and has some of same consequences, 
such as the failure of the approachability property on $\omega_1$. 
It follows that $\mathsf{ISP}(\omega_2)$ 
implies that $2^\omega \ge \omega_2$. 

Viale and Weiss \cite{weiss1} asked whether this principle settles the value of 
the continuum. 
We solve this problem by showing that $\mathsf{ISP}(\omega_2)$ 
is consistent with $2^\omega$ being arbitrarily large. 
The solution is an application of the quotient theorem described above 
and the second author's method of adequate set forcing (\cite{jk21}).

\bigskip

Section 1 provides background on regular suborders and quotients, 
as well as guessing models and $\mathsf{ISP}$. 
Section 2 introduces simple universal strongly generic conditions 
and proves the main result on quotients. 
Section 3 discusses products of strongly proper forcings. 
Section 4 provides two variations of the main quotient theorem. 
Section 5 gives an example showing that not all quotients of strongly 
proper forcings are well behaved. 
Section 6 describes a strongly proper collapse using the method of 
adequate set forcing. 
Section 7 constructs a model in which $\mathsf{ISP}(\omega_2)$ holds and 
$2^\omega$ is arbitrarily large.

\section{Background}

In this section we provide the background necessary for reading the paper. 
We assume that the reader is already familiar with the basics of proper forcing and 
generalized stationarity. 
First, we will review some well-known results about 
regular suborders and quotients; since these ideas 
are central to the paper we provide a thorough treatment. 
Secondly, we review the idea of a guessing model, the approximation property, 
and the principle $\mathsf{ISP}(\omega_2)$.

For the remainder of the section fix a forcing poset $\q$ which has 
greatest lower bounds. 
In other words, for all compatible conditions $p$ and $q$ in $\q$, 
a greatest lower bound $p \wedge q$ exists. 
A \emph{suborder} of $\q$ is a set $\p \subseteq \q$ 
ordered by $\le_\p := \le_\q \cap (\p \times \p)$. 
A suborder $\p$ of $\q$ is said to be a \emph{regular suborder} if 
(a) for all $p$ and $q$ in $\p$, if $p$ and $q$ are compatible in $\q$ 
then $p$ and $q$ are compatible in $\p$, and (b) 
if $A$ is a maximal antichain of $\p$, then $A$ is a maximal 
antichain of $\q$.

\begin{lemma}
Let $\p$ be a regular suborder of $\q$. 
Then for all $q \in \q$, there is $s \in \p$ such that for all 
$t \le s$ in $\p$, $q$ and $t$ are compatible in $\q$.
\end{lemma}

\begin{proof}
Suppose for a contradiction that there is $q$ in $\q$ such that for all 
$s \in \p$, there is $t \le s$ in $\p$ 
such that $t$ is incompatible with $q$ in $\q$. 
Let $D$ be the set of $t \in \p$ such that $t$ 
is incompatible with $q$ in $\q$. 
Then $D$ is dense in $\p$. 
Let $A$ be a maximal antichain of $\p$ contained in $D$. 
Since $\p$ is a regular suborder, $A$ is maximal in $\q$. 
Therefore $q$ is compatible with some member of $A$, 
which contradicts the definition of $D$.
\end{proof}

\begin{definition}
Let $\p$ be a regular suborder of $\q$. 
Then $\q / \dot G_\p$ is a $\p$-name for the poset consisting of conditions 
$q \in \q$ such that for all $s \in \dot G_\p$, $q$ and $s$ are compatible in $\q$, 
with the same ordering as $\q$.
\end{definition}

Note that if $p * \check q$ is in $\p * (\q / \dot G_\p)$, then $p$ and $q$ 
are compatible, so $p \wedge q$ exists.

\begin{lemma}
Let $\p$ be a regular suborder of $\q$. 
Let $q \in \q$ and $s \in \p$. 
Then the following are equivalent:
\begin{enumerate}
\item $s \Vdash_\p q \in \q / \dot G_\p$;
\item for all $t \le s$ in $\p$, $q$ and $t$ are compatible in $\q$.
\end{enumerate}
\end{lemma}

The proof is straightforward. 

\begin{lemma}
Let $\p$ be a regular suborder of $\q$. 
Then $\p$ forces that whenever $q \in \q / \dot G_\p$ and $q \le p$, 
then $p \in \q / \dot G_\p$.
\end{lemma}

The proof is easy.
 
\begin{lemma}
Let $\p$ be a regular suborder of $\q$. 
If $D$ is a dense subset of $\q$, then $\p$ forces that 
$D \cap (\q / \dot G_\p)$ is a dense subset of $\q / \dot G_\p$.
\end{lemma}

\begin{proof}
Suppose for a contradiction that $p \in \p$ and $p$ forces that 
$q$ is in $\q / \dot G_\p$ but $q$ has no extension in 
$D \cap (\q / \dot G_\p)$. 
Then $p$ is compatible with $q$. 
Fix $r \le q \wedge p$ in $D$. 
Apply Lemma 1.1 to find $v$ in $\p$ such that every extension of $v$ 
in $\p$ is compatible with $r$. 
In particular, $v$ is compatible with $r$ and hence with $p$. 
Since $v$ and $p$ are in $\p$ and $\p$ is a regular suborder, 
$v$ and $p$ are compatible in $\p$. 
So without loss of generality assume that $v \le p$. 
By Lemma 1.3, $v$ forces that $r$ is in $\q / \dot G_\p$. 
Since $r \in D$, $v$ forces that $r$ is an extension of $q$ in 
$D \cap (\q / \dot G_\p)$, contradicting that $v \le p$.
\end{proof}

\begin{lemma}
Let $\p$ be a regular suborder of $\q$.
\begin{enumerate}
\item Suppose that $H$ is a $V$-generic filter on $\q$. 
Then $H \cap \p$ is a $V$-generic filter on $\p$ and 
$H$ is a $V[H \cap \p]$-generic filter on $\q / (H \cap \p)$. 
\item Suppose that $G$ is a $V$-generic filter on $\p$ and $H$ is a 
$V[G]$-generic filter on $\q / G$. 
Then $H$ is a $V$-generic filter on $\q$, 
$G = H \cap \p$, and $V[G][H] = V[H]$.
\end{enumerate}
\end{lemma}

\begin{proof}
(1) Suppose that $H$ is a $V$-generic filter on $\q$. 
Since any maximal antichain of $\p$ is a maximal antichain of $\q$, 
$H \cap \p$ meets every maximal antichain of $\p$. 
A straightforward density argument shows that $H \cap \p$ 
is a filter. 
So $H \cap \p$ is a $V$-generic filter on $\p$.

Since $H$ is a filter, every member of $H$ is compatible in $\q$ 
with every member of $H \cap \p$. 
So $H \subseteq \q / (H \cap \p)$ and $H$ is a filter on $\q / (H \cap \p)$. 
We will show that $H$ meets every dense subset 
of $\q / (H \cap \p)$ in $V[H \cap \p]$. 

Let $\dot D$ be a $\p$-name for a dense subset of $\q / \dot G_\p$ and 
let $D := \dot D^{H \cap \p}$. 
We will show that $D \cap H \ne \emptyset$. 
Let $D'$ be the set of conditions in $\q$ of the form 
$p \wedge q$, where $p$ is in $\p$ and $p$ forces that 
$q$ is in $\dot D$. 
Then $D'$ is dense in $\q$ by a straightforward argument 
using Lemmas 1.1 and 1.3. 
Fix $p \wedge q$ in $D' \cap H$. 
Then $q \in H$, and since $p \in H \cap \p$, 
$q \in D$.

(2) By Lemma 1.4, $H$ is closed upwards in $\q$, so $H$ is a filter on $\q$. 
Let $D$ be a dense subset of $\q$ in $V$, and we will show that 
$H \cap D \ne \emptyset$. 
By Lemma 1.5, $D \cap (\q / G)$ is a dense subset of $\q / G$. 
Since $H$ is a $V[G]$-generic filter on $\q / G$, it meets $D \cap (\q / G)$ 
and hence $D$.

To show that $V[G][H] = V[H]$, it suffices to show that 
$G = H \cap \p$. 
Since $G$ and $H \cap \p$ are both $V$-generic filters on $\p$ by (1), it
suffices to show that every condition in $G$ is compatible with 
every condition in $H \cap \p$. 
But if $q \in H \cap \p$, then $q \in \q / G$, which implies that $q$ is 
compatible in $\q$ with every condition in $G$. 
Since $\p$ is a regular suborder of $\q$, 
$q$ is compatible in $\p$ with every condition in $G$.
\end{proof}

It follows that $\q$ is forcing equivalent to $\p * (\q / \dot G_\p)$. 
In fact, the function which sends a condition in $\p * (\q / \dot G_\p)$ of 
the form $p * \check q$ to $p \wedge q$ is a dense embedding defined 
on a dense subset of $\p * (\q / \dot G_\p)$. 

The next technical lemma will be used later in the paper.

\begin{lemma}
Let $\p$ be a regular suborder of $\q$. 
Suppose that $q \in \q$, $p \in \p$, and $p$ forces in $\p$ that 
$q$ is not in $\q / \dot G_\p$. Then $p$ and $q$ are incompatible.
\end{lemma}

\begin{proof}
Suppose for a contradiction that $p$ forces that $q$ is not in 
$\q / \dot G_\p$, but there is $r \le p, q$ in $\q$. 
Let $H$ be a $V$-generic filter on $\q$ such that $r \in H$. 
Then $p$ and $q$ are in $H$. 
Since $\p$ is a regular suborder of $\q$, 
$H \cap \p$ is a $V$-generic filter on $\p$ by Lemma 1.6. 
As $p \in H \cap \p$, $q$ is not in $\q / (H \cap \p)$. 
Therefore $q$ is incompatible in $\q$ with some member of $H \cap \p$. 
But this is impossible since $q \in H$ and $H$ is a filter.
\end{proof}

\bigskip

We now provide the necessary background on guessing models, the 
approximation property, and $\mathsf{ISP}$.

\begin{definition}
A set $N$ is said to be \emph{$\omega_1$-guessing} 
if for any set of ordinals 
$d \subseteq N$ such that $\sup(d) < \sup(N \cap On)$, 
if $d$ satisfies that for any 
countable set $b \in N$, $d \cap b \in N$, then there exists $d' \in N$ 
such that $d = d' \cap N$.
\end{definition}

\begin{definition}
Let $W_1$ and $W_2$ be transitive with $W_1 \subseteq W_2$. 
We say that the pair $(W_1,W_2)$ has the 
\emph{$\omega_1$-approximation property} 
if whenever $d \in W_2$ is a bounded 
subset of $W_1 \cap On$ and satisfies that $b \cap d \in W_1$ 
for any set $b \in W_1$ which is countable in $W_1$, 
then $d \in W_1$.
\end{definition}

\begin{lemma}
Let $N$ be an elementary substructure of $H(\chi)$ for some uncountable 
cardinal $\chi$. 
Then the following are equivalent:
\begin{enumerate}
\item $N$ is an $\omega_1$-guessing model;
\item the pair $(\overline N,V)$ has the $\omega_1$-approximation 
property, where $\overline N$ is the transitive collapse of $N$.
\end{enumerate}
\end{lemma}

\begin{proof}
Let $\sigma : N \to \overline N$ be the transitive collapsing map.

$(1 \Rightarrow 2)$ 
Assume that $N$ is an $\omega_1$-guessing model. 
To show that $(\overline N,V)$ has the $\omega_1$-approximation property, 
let $d$ be a bounded subset of $\overline N \cap On$ and assume that 
whenever $\overline N$ models that $b$ is a countable set of ordinals, 
$b \cap d \in \overline N$. 
We will prove that $d \in \overline N$.

Let $e := \sigma^{-1}[d]$. 
Then $e$ is a subset of $N \cap On$ and $\sup(e) < \sup(N \cap On)$. 
Let $b \in N$ be a countable set of ordinals, and we show that 
$e \cap b \in N$. 
Since $N \prec H(\chi)$, $N$ models that $b$ is countable. 
Therefore $\overline N$ models that $\sigma(b)$ is countable. 
It follows that $d \cap \sigma(b) \in \overline N$. 
Hence $\sigma^{-1}(d \cap \sigma(b)) \in N$. 
Since $b$ and $d \cap \sigma(b)$ are countable, 
$\sigma^{-1}(d \cap \sigma(b)) = 
\sigma^{-1}[d \cap \sigma[b]] = 
\sigma^{-1}[d] \cap \sigma^{-1}[\sigma[b]] = e \cap b$. 
Therefore $e \cap b$ is in $N$.

Since $N$ is an $\omega_1$-guessing model, fix $e' \in N$ such that 
$e = e' \cap N$. 
Then $\sigma(e') \in \overline{N}$ and 
$\sigma(e') = \sigma[e' \cap N] = \sigma[e] = d$. 
It follows that $d$ is in $\overline N$, as desired.

$(2 \Rightarrow 1)$ 
Suppose that $(\overline N,V)$ has the $\omega_1$-approximation property. 
Let $d \subseteq N \cap On$ be given with $\sup(d) < \sup(N \cap On)$, 
and assume that for any 
countable set $b \in N$, $d \cap b \in N$. 
Let $e := \sigma[d]$, which is a bounded subset of $\overline N \cap On$. 
Suppose that $\overline N$ models that $b$ is a countable set of ordinals. 
Then $\sigma^{-1}(b)$ is a countable set in $N$, 
so $d \cap \sigma^{-1}(b) \in N$. 
Since $b$ and $d \cap \sigma^{-1}(b)$ are countable, 
$\sigma(d \cap \sigma^{-1}(b)) = \sigma[d \cap \sigma^{-1}[b]] = 
\sigma[d] \cap b = e \cap b$. 
Hence $e \cap b$ is in $\overline N$. 
By the $\omega_1$-approximation property, $e$ is in $\overline N$. 
Hence $d' := \sigma^{-1}(e)$ is in $N$. 
But $d = d' \cap N$, as can be easily checked.
\end{proof}

\begin{definition}
We say that the principle $\mathsf{ISP}(\omega_2)$ holds if 
for all sufficiently large regular uncountable cardinals $\chi$, there are 
stationarily many $N$ in $P_{\omega_2}(H(\chi))$ such that 
$N \cap \omega_2 \in \omega_2$, $N \prec H(\chi)$, and 
$N$ is an $\omega_1$-guessing model. 
\end{definition}

The principle $\mathsf{ISP}(\omega_2)$ was introduced by 
Weiss (\cite{weiss1}, \cite{weiss2}), in a different form which asserts the 
existence of an ineffable branch for any slender $P_{\omega_2}(\lambda)$-list, 
for all cardinals $\lambda \ge \omega_2$. 
Roughly speaking, $\mathsf{ISP}(\omega_2,\lambda)$ is a 
$P_{\omega_2}(\lambda)$ version of the tree property, where usual the 
tree property on $\omega_2$ states that every $\omega_2$-tree 
has a cofinal branch. 
The equivalence of this principle with Definition 1.11 is proven 
in \cite[Section 3]{weiss1}. 
See \cite[Section 2]{weiss1} for the alternative definition 
of $\mathsf{ISP}$ and a discussion of the principle. 

\section{Quotients of Strongly Proper Forcings}

In this section we prove our main theorem on quotients of 
strongly proper forcings. 
We introduce the idea of a simple universal strongly $N$-generic condition,  
and show that under some circumstances quotients of forcing posets 
which have such conditions for stationarily many $N$ are well behaved.
 
The following definition was introduced by Mitchell in a slightly 
different form (\cite{mitchell2}). 

\begin{definition}
Let $\q$ be a forcing poset, $q \in \q$, and $N$ a set. 
We say that $q$ is a \emph{strongly $(N,\q)$-generic condition} 
if for any set $D$ which is dense in the forcing poset $N \cap \q$, 
$D$ is predense in $\q$ below $q$.
\end{definition}

If $\q$ is understood from context, we just say that $q$ is a 
strongly $N$-generic condition.

For a forcing poset $\q$, let $\lambda_\q$ 
denote the smallest uncountable cardinal $\lambda$ such that 
$\q \subseteq H(\lambda)$. 
Note that for any set $N$, a condition $q \in \q$ 
is strongly $N$-generic iff 
$q$ is strongly $(N \cap H(\lambda_\q))$-generic.

The next lemma was basically proven in \cite[Proposition 2.15]{mitchell2}. 
We include a proof for completeness.

\begin{lemma}
Let $\q$ be a forcing poset, $q \in \q$, and $N$ a set. 
Then the following are equivalent:
\begin{enumerate}
\item $q$ is strongly $N$-generic;
\item there is a function $r \mapsto r \restriction N$ defined on 
conditions $r \le q$ satisfying that $r \restriction N \in N \cap \q$ 
and for all $v \le r \restriction N$ in $N \cap \q$, $r$ and $v$ are compatible. 
\end{enumerate}
\end{lemma}

\begin{proof}
For the forward direction, suppose that there is 
$r \le q$ for which there does not exist a condition 
$r \restriction N$ all of whose extensions in $N \cap \q$ 
are compatible with $r$. 
Let $D$ be the set of $w \in N \cap \q$ which are incompatible with $r$. 
The assumption on $r$ implies that $D$ is dense in $N \cap \q$. 
But $D$ is not predense below $q$ since every condition in $D$ is 
incompatible with $r$. 
So $q$ is not strongly $N$-generic.

Conversely, assume that there is a function 
$r \mapsto r \restriction N$ as described. 
Let $D$ be dense in $N \cap \q$ and let $r \le q$. 
Fix $v \le r \restriction N$ in $D$. 
Then $r$ and $v$ are compatible. 
So $D$ is predense below $q$.
\end{proof}

We introduce two strengthenings of strong genericity, namely simple and universal.

\begin{definition}
Let $\q$ be a forcing poset, $q \in \q$, and $N$ a set. 
We say that $q$ is a \emph{universal strongly $(N,\q)$-generic condition} if 
$q$ is a strongly $(N,\q)$-generic condition and for all 
$p \in N \cap \q$, $p$ and $q$ are compatible.
\end{definition}

\begin{definition}
Let $\q$ be a forcing poset, $q \in \q$, and $N$ a set. 
We say that $q$ is a \emph{simple strongly $(N,\q)$-generic condition} if 
there exists a set $E \subseteq \q$ which is 
dense below $q$ and a function $r \mapsto r \restriction N$ 
defined on $E$ such that for all $r \in E$, 
$r \restriction N \in N \cap \q$, $r \le r \restriction N$, 
and for all  
$v \le r \restriction N$ in $N \cap \q$, $r$ and $v$ are compatible.
\end{definition}

The difference between simple and ordinary strongly generic conditions as 
described in Lemma 2.2(2) is the additional assumption in simple that 
$r \le r \restriction N$. 
But if the forcing poset has greatest lower bounds, then the two ideas are 
equivalent; 
see Lemma 2.5 below.

Note that if $q$ is a simple strongly $N$-generic condition, then 
$q$ is strongly $N$-generic. 
Namely, given $r \le q$ not in $E$, first extend $r$ to $s$ in $E$ 
and then define 
$r \restriction N$ to be $s \restriction N$.

While these definitions are new, most strongly generic conditions in 
the literature satisfy them. 
For example, all of the adequate set type forcings described in \cite{jk21} 
have simple universal strongly generic conditions for countable 
elementary substructures.

It turns out that for forcing posets with greatest lower bounds, every strongly 
generic condition is simple.

\begin{lemma}
Let $\q$ be a forcing poset with greatest lower bounds. 
Let $N$ be a set and $q \in \q$. 
If $q$ is a strongly $N$-generic condition, then $q$ is a simple strongly $N$-generic 
condition.
\end{lemma}

\begin{proof}
Let $r_0 \le q$, and we will find $r \le r_0$ and $r \restriction N \in N \cap \q$ 
such that $r \le r \restriction N$, and for all $v \le r \restriction N$ in 
$N \cap \q$, $r$ and $v$ are compatible. 
Since $q$ is strongly $N$-generic, there is a condition $u \in N \cap \q$ such that 
for all $v \le u$ in $N \cap \q$, $r_0$ and $v$ are compatible. 
In particular, $r_0$ and $u$ are compatible. 
Let $r := r_0 \land u$ and $r \restriction N := u$. 
Then $r \le r_0$, $r \restriction N \in N \cap \q$, and $r \le r \restriction N$.

Let $v \le r \restriction N$ be in $N \cap \q$, and we will show that 
$r$ and $v$ are compatible. 
Then $v \le u$, so by the choice of $u$, $r_0$ and $v$ are compatible. 
Therefore $r_0 \land v$ exists. 
But $r_0 \land v \le v$, and 
since $v \le u$, $r_0 \land v \le r_0 \land u = r$. 
So $r$ and $v$ are compatible.
\end{proof}

\begin{definition}
A forcing poset $\q$ is \emph{strongly proper} if for all large enough regular 
cardinals $\chi$, there are club many countable elementary substructures 
$N$ of $H(\chi)$ satisfying that whenever $p \in N \cap \q$, 
there exists a strongly $N$-generic condition below $p$.
\end{definition}

This definition, introduced by Mitchell, is defined in a way similar to 
the usual definition of proper forcing. 
But by standard arguments and the comments after Definition 2.1, 
this definition is equivalent to the existence of 
club many countable sets $N$ in 
$P_{\omega_1}(H(\lambda_\q))$ 
such that every $p \in N \cap \q$ has a strongly 
$N$-generic extension.

\begin{definition}
Let $\q$ be a forcing poset. 
We say that $\q$ is \emph{strongly proper on a stationary set} if there are 
stationarily many $N$ in $P_{\omega_1}(H(\lambda_\q))$ 
such that whenever $p \in N \cap \q$, 
there is $q \le p$ which is a strongly $N$-generic condition.
\end{definition}

By standard arguments we get an equivalent property by replacing $\lambda_\q$ with any 
regular cardinal $\lambda \ge \lambda_\q$; similar comments apply to the 
following definition.

\begin{definition}
A forcing poset $\q$ is said to 
\emph{have universal strongly generic conditions on a 
stationary set} if there are 
stationarily many $N$ in $P_{\omega_1}(H(\lambda_\q))$ 
such that there exists a universal strongly $N$-generic condition.
\end{definition}

For a forcing poset $\q$ with greatest lower bounds, 
we will consider suborders $\p$ of $\q$ satisfying the following 
compatibility property. 
This property will be crucial for the rest of the paper.

\bigskip

\noindent $*(\p,\q)$: for all $x$ in $\p$ and $y$, $z$ in $\q$, 
if $x$, $y$, and $z$ are pairwise compatible, 
then $x$ is compatible with $y \wedge z$.

\bigskip

Note that $*(\q,\q)$ implies that $*(\p,\q)$ for all suborders $\p$ of $\q$. 
A large number of forcing posets $\q$ satisfy the property $*(\q,\q)$. 
These include Cohen forcings, collapsing forcings, adding a club by 
initial segments, and many side condition forcings. 
In general, forcing posets where greatest lower bounds are given by unions 
tend to satisfy it. 
In contrast, most Boolean algebras do not satisfy the property. 
For example, consider a field of sets, and let $A$ and $B$ be sets with 
nonempty intersection and relative complements. 
Then $A$, $B$, and $A \triangle B$ have pairwise nonempty intersections, 
but $A$ has empty intersection with $B \cap (A \triangle B) = B \setminus A$.

Our main theorem on quotients Theorem 2.11 
states that if $\q$ has universal 
strongly generic conditions on a stationary set, and $\p$ is a regular 
suborder satisfying $*(\p,\q)$, 
then $\p$ forces that the quotient $\q / G_{\dot \p}$ has 
universal strongly generic conditions on a stationary set.

\begin{lemma}
Let $\q$ be a forcing poset which has greatest lower bounds, and let 
$\p$ be a regular suborder of $\q$ which satisfies $*(\p,\q)$. 
Then $\p$ forces that for all $r$ and $s$ in $\q / \dot G_\p$, if 
$r$ and $s$ are compatible in $\q$ then $r \wedge s$ is in $\q / \dot G_\p$.
\end{lemma}

\begin{proof}
Let $G$ be a $V$-generic filter on $\p$, and 
suppose that $r$ and $s$ are in $\q / G$ and are compatible in $\q$. 
To show that $r \wedge s$ is in $\q / G$, let $p$ be in $G$ and we will show 
that $p$ is compatible with $r \wedge s$. 
Since $r$ and $s$ are in $\q / G$, they are each compatible with $p$. 
By property $*(\p,\q)$, $p$ is compatible with $r \wedge s$.
\end{proof}

The next lemma will be used in the proof of Theorem 2.11, which is our 
main theorem on quotients. 
We will use it again in Section 4 when we prove a variation of the main theorem.

\begin{lemma}
Let $\q$ be a forcing poset which has greatest lower bounds, and let 
$\p$ be a regular suborder of $\q$ satisfying property $*(\p,\q)$. 
Let $\chi$ be a regular cardinal, $N \prec H(\chi)$, $\p, \q \in N$, and suppose that 
$q$ is a strongly $(N,\q)$-generic condition. 
Assume that $z$ is in $\p$ and $z$ forces that $\check q$ is in $\q / \dot G_\p$. 
Then $z$ forces that $N[\dot G_\p] \cap V \subseteq N$ and 
$\check q$ is a strongly 
$(N[\dot G_\p],\q / \dot G_\p)$-generic condition.
\end{lemma}

\begin{proof}
To show that $z$ forces that $N[\dot G_\p] \cap V \subseteq N$, 
assume for a contradiction that 
$z' \le z$ in $\p$, $\dot a$ is a $\p$-name in $N$, and 
$z'$ forces that $\dot a$ is in $V \setminus N$. 
Since $\dot a^G = \dot a^{G \cap \p}$ 
whenever $G$ is a $V$-generic filter on $\q$, 
it follows that $z'$ forces in $\q$ that $\dot a$ is in $V \setminus N$. 
Let $D$ be the dense set of conditions 
$v \in \q$ such that $v$ decides in $\q$ 
whether $\dot a$ is in $V$, and if it decides that it is, 
it decides the value of $\dot a$. 
Since $\dot a \in N$, it follows by elementarity that 
$D \in N$, and hence $D \cap N$ is dense in $N \cap \q$. 

Since $z$ forces that $\check q$ is in $\q / \dot G_\p$, 
$z'$ is compatible with $q$ in $\q$. 
As $q$ is strongly $(N,\q)$-generic and $q \wedge z' \le q$, 
fix $v \in D \cap N$ which is compatible with $q \wedge z'$. 
As $z'$ forces that $\dot a$ is in $V \setminus N$, 
so does $v \wedge q \wedge z'$. 
Since $v \in D \cap N$, by the elementarity of $N$ 
there is $b \in N$ such that $v$ forces that $\dot a = \check b$. 
But then $v \wedge q \wedge z'$ forces that $\dot a \in N$, which 
is a contradiction.

Now let $G_\p$ be a $V$-generic filter on $\p$ which contains $z$.
We will show that in $V[G_\p]$, $q$ is a strongly 
$(N[G_\p], \q / G_\p)$-generic condition. 
Since $q$ is a strongly $(N,\q)$-generic condition in $V$, 
by Lemma 2.5 fix a set $E \subseteq \q$ in $V$ which is dense below $q$ 
and a function $r \mapsto r \restriction N$ defined on $E$ 
satisfying that for all $r \in E$, $r \restriction N \in N \cap \q$, 
$r \le r \restriction N$, and for all $v \le r \restriction N$ in $N \cap \q$, 
$r$ and $v$ are compatible. 
By Lemma 1.5, $E \cap (\q / G_\p)$ 
is dense below $q$ in $\q / G_\p$.

In $V[G_\p]$ define a map $r \mapsto r \restriction N[G_\p]$ on the set 
$E \cap (\q / G_\p)$ by letting $r \restriction N[G_\p] := r \restriction N$. 
For all $r \in E \cap (\q / G_\p)$, $r \restriction N[G_\p]$ 
is in $N$ and hence in $N[G_\p]$, and 
$r \le r \restriction N = r \restriction N[G_\p]$. 
By Lemma 1.4, $r \restriction N[G_\p]$ is in $\q / G_\p$. 
Consider $v \in N[G_\p] \cap (\q / G_\p)$ below $r \restriction N[G_\p]$. 
Since $N[G_\p] \cap V \subseteq N$ as shown above, 
$v \in N \cap \q$ and $v \le r \restriction N$ in $\q$. 
So $v$ is compatible with $r$ in $\q$. 
As $r$ and $v$ are both in $\q / G_\p$, $r \wedge v$ is in $\q / G_\p$ 
by Lemma 2.9. 
Hence $r$ and $v$ are compatible in $\q / G_\p$.
\end{proof}

\begin{thm}
Let $\q$ be a forcing poset which has greatest lower bounds, and let 
$\p$ be a regular suborder of $\q$ satisfying property $*(\p,\q)$. 
Assume that $\q$ has universal strongly generic conditions 
on a stationary set. 
Then $\p$ forces that $\q / \dot G_\p$ has universal strongly generic 
conditions on a stationary set. 
In particular, $\p$ forces that $\q / \dot G_\p$ 
is strongly proper on a stationary set.
\end{thm}

\begin{proof}
Let $\chi$ be a regular cardinal such that $\q \in H(\chi)$. 
Then $\q$ forces that $\chi$ is regular and 
$\lambda_{\q / \dot G_\p} \le \chi$. 
Let $\dot F$ be a $\p$-name for a function 
$\dot F : (H(\chi)^{V[\dot G_\p]})^{<\omega} \to H(\chi)^{V[\dot G_\p]}$, 
and let $s \in \p$.  
We will find an extension of $s$ in $\p$ which forces that there exists a countable 
set $M \subseteq H(\chi)^{V[\dot G_\p]}$ 
which is closed under $\dot F$ such that there exists a 
universal strongly $(M,\q / \dot G_\p)$-generic condition.

Define $H : H(\chi)^{<\omega} \to H(\chi)$ by letting 
$H(\dot a_0,\ldots,\dot a_n)$ be a $\p$-name in $H(\chi)$ 
which $\p$ forces is equal to $\dot F(\dot a_0,\ldots,\dot a_n)$, for any 
$\p$-names $\dot a_0, \ldots, \dot a_n$ in $H(\chi)$. 
Since $\q$ has universal strongly generic conditions on a stationary set, 
we can fix a countable set $N \prec H(\chi)$ 
such that $\p$, $\q$, and $s$ are in $N$, $N$ is closed under $H$, 
and there is a universal strongly $(N,\q)$-generic condition $q_N$. 

Since $s \in N \cap \q$ and $q_N$ is universal, 
$q_N$ is compatible with $s$. 
As $\p$ is a regular suborder of $\q$, fix $z$ in $\p$ 
such that every extension of $z$ in $\p$ is compatible with 
$q_N \wedge s$ in $\q$. 
By Lemma 1.3, $z$ forces that $q_N \wedge s$ is in $\q / \dot G_\p$. 
By Lemma 1.4, $z$ forces that $q_N$ is in $\q / \dot G_\p$. 
Since $z$ is compatible with $q_N \wedge s$ in $\q$, $z$ is 
compatible with $s$ in $\q$. 
As $\p$ is a regular suborder of $\q$ and $s$ and $z$ are in $\p$, 
$z$ is compatible with $s$ in $\p$. 
Extending $z$ if necessary in $\p$, we may assume without loss 
of generality that $z \le s$. 

Since $N$ is closed under $H$, $\p$ forces that 
$N[\dot G_\p]$ is closed under $\dot F$. 
So it suffices to show that $z$ forces that $q_N$ is a universal 
strongly $(N[\dot G_\p],\q / \dot G_\p)$-generic condition. 
Let $G$ be a $V$-generic filter on $\p$ with $z \in G$. 
By Lemma 2.10 applied to $z$ and $q_N$, we get that 
$N[G] \cap V \subseteq N$ and $q_N$ is a 
strongly $(N[\dot G_\p],\q / \dot G_\p)$-generic condition. 
Finally, if $p \in N[G] \cap (\q / G)$, then $p \in N \cap \q$, so 
$p$ and $q_N$ are compatible in $\q$ by the universality of $q_N$. 
As $p$ and $q_N$ are in $\q / G$, they are compatible in 
$\q / G$ by Lemma 2.9.
\end{proof}

Recall the following definition of Mitchell, which plays a prominent role 
in \cite{mitchell2}. 

\begin{definition}
A forcing poset $\q$ is said to have the 
\emph{$\omega_1$-approximation property} 
if $\q$ forces that whenever $X$ is a subset of $V$ such that for every set $N$ which is countable in $V$, 
$N \cap X$ is in $V$, then $X$ is in $V$.
\end{definition}

In other words, $\q$ has the $\omega_1$-approximation property iff 
$\q$ forces that the pair $(V,V[\dot G_\q])$ has the 
$\omega_1$-approximation property.

\begin{proposition}
Let $\q$ be a forcing poset and assume that 
$\q$ is strongly proper on a stationary set. 
Then $\q$ satisfies the $\omega_1$-approximation property.
\end{proposition}

Proposition 2.13 is a special case of \cite[Lemma 6]{mitchell1}. 
The proof goes roughly as follows. 
Suppose for a contradiction that 
a condition $p$ forces that $\dot X$ is a counterexample to the 
approximation property. 
Fix $N$ an elementary substructure 
with $p, \dot X \in N$ and $q \le p$ which is a strongly $N$-generic condition. 
Extend $q$ to $r$ which decides the value of $N \cap \dot X$. 
Then $r \restriction N$ has extensions $v$ and $w$ in $N$ which disagree on whether 
a certain set in $N$ is in $\dot X$. 
Since $r$ decides the value of $N \cap \dot X$, it cannot be compatible 
with both $v$ and $w$, giving a contradiction.

As an immediate consequence of Theorem 2.11 and Proposition 2.13, we 
get the following corollary.

\begin{corollary}
Let $\q$ be a forcing poset which has greatest lower bounds, 
and assume that $\q$ has 
universal strongly generic conditions on a stationary set. 
Let $\p$ be a regular suborder of $\q$ satisfying $*(\p,\q)$. 
Then $\p$ forces that $\q / \dot G_\p$ satisfies the 
$\omega_1$-approximation property.
\end{corollary}

\section{Products}

We will show that some of the properties studied in the previous section 
are preserved under products.\footnote{Similar results were obtained previously 
by Friedman \cite[Lemma 3]{friedman2} and Neeman 
\cite[Claim 3.8]{neeman}.}  
This information will be used in our consistency result on 
$\mathsf{ISP(\omega_2)}$.

Note that if $\p$ and $\q$ have greatest lower bounds, then so does $\p \times \q$. 
Namely, $(p_0,q_0) \wedge (p_1,q_1) = (p_0 \wedge p_1,q_0 \wedge q_1)$.

\begin{lemma}
Suppose that $\p$ and $\q$ are forcing posets with greatest lower bounds which 
both satisfy the property that whenever $x$, $y$, and $z$ are pairwise compatible 
conditions, then $x$ is compatible with $y \wedge z$. 
Then $\p \times \q$ satisfies this property as well.
\end{lemma}

The proof is straightforward.

\begin{lemma}
Let $\p$ and $\q$ be forcing posets with greatest lower bounds. 
Let $\lambda$ be a cardinal such that $\lambda_\p, \lambda_\q \le \lambda$ and 
let $S$ be a subset of $P_{\omega_1}(H(\lambda))$. 
Suppose that $\p$ and $\q$ 
both have universal strongly generic conditions on $S$. 
Then $\p \times \q$ has universal strongly generic conditions on $S$.
\end{lemma}

\begin{proof}
Let $N$ be in $S$. 
Fix a universal 
strongly $(N,\p)$-generic condition $p$ and a universal 
strongly $(N,\q)$-generic condition $q$. 
We claim that $(p,q)$ is a universal strongly $(N,\p \times \q)$-generic condition. 

So let $(u,v)$ be in $N \cap (\p \times \q)$. 
Then by universality, 
$p$ and $u$ are compatible in $\p$, and 
$q$ and $v$ are compatible in $\q$. 
Hence $(p \wedge u,q \wedge v) \le (p,q), (u,v)$. 
So $(p,q)$ and $(u,v)$ are compatible in $\p \times \q$.

Let $r \mapsto r \restriction N$ be a map from conditions $r \le p$ in $\p$ 
to $N \cap \p$ witnessing that $p$ is a 
strongly $(N,\p)$-generic condition. 
Similarly, let $s \mapsto s \restriction N$ be a map from conditions $s \le q$ in $\q$ 
to $N \cap \q$ 
witnessing that 
$q$ is a strongly $(N,\q)$-generic condition.

For $(r,s) \le (p,q)$ in $\p \times \q$, 
define $(r,s) \restriction N = 
(r \restriction N, s \restriction N)$, which is clearly 
in $N \cap (\p \times \q)$. 
Assume that $(y,z) \le (r \restriction N,s \restriction N)$ 
is in $N \cap (\p \times \q)$. 
Then $y$ is compatible with $r$ in $\p$ and $z$ is compatible 
with $s$ in $\q$. 
So $(r \wedge y,s \wedge z)$ is in $\p \times \q$ and is below 
$(r,s)$ and $(y,z)$. 
So every extension of $(r,s) \restriction N$ in $N \cap (\p \times \q)$ 
is compatible with $(r,s)$.
\end{proof}

A similar argument shows that if $\p$ and $\q$ are strongly proper, 
then so is $\p \times \q$. 
In other words, strong properness is productive. 
This is in contrast to properness, which is not productive; 
see \cite[Chapter XVII, 2.12]{shelah}.

\section{Variations}

We consider two variations of Theorem 2.11 on quotients of strongly proper forcings. 
First, we discuss factoring a forcing poset over an 
elementary substructure below a condition. 
Secondly, we introduce a weakening of strongly proper which is called 
non-diagonally strongly proper, and show that this property is sometimes 
preserved under taking quotients.

\begin{definition}
Let $\p$ and $\q$ be forcing posets. 
A function $f : \p \to \q$ is a \emph{regular embedding} if:
\begin{enumerate}
\item for all $p$ and $q$ in $\p$, $q \le p$ implies $f(q) \le f(p)$;
\item for all $p$ and $q$ in $\p$, if $f(p)$ and $f(q)$ are compatible in $\q$, 
then $p$ and $q$ are compatible in $\p$;
\item if $A$ is a maximal antichain of $\p$, then $f[A]$ is a maximal 
antichain of $\q$.
\end{enumerate}
\end{definition}

It is straightforward to show that 
if $f : \p \to \q$ is a regular embedding, then 
$f[\p]$ is a regular suborder of $\q$.

Previously we have focused on strongly generic conditions for countable models. 
The next lemma and theorem are useful when the model under 
consideration is uncountable.

Recall that $\p / q = \{ r \in \p : r \le q \}$, where $\p$ is a forcing poset 
and $q \in \p$.

\begin{lemma}
Let $\p$ be a forcing poset with greatest lower bounds,  
$\chi \ge \lambda_\p$ a regular cardinal, and 
$N \prec (H(\chi),\in,\p)$. 
Suppose that $q$ is a universal strongly $N$-generic condition. 
Applying the universality of $q$, define a function 
$f : (N \cap \p) \to (\p / q)$ by $f(p) = q \wedge p$.

Then $f$ is a regular embedding. 
Moreover, for any $V$-generic filter $G$ on $N \cap \p$, 
for all $r \in \p / q$, $r \in (\p / q) / f[G]$ iff 
$r$ is compatible with every condition in $G$. 
\end{lemma}

\begin{proof}
We prove first that $f$ is a regular embedding. 
If $t \le s$ in $N \cap \p$, then easily $q \wedge t \le q \wedge s$. 
Now let $s$ and $t$ be in $N \cap \p$ and assume that 
$w \le q \wedge s, q \wedge t$. 
Then $w \le s, t$. 
By the elementarity of $N$, there is a condition $w'$ in $N \cap \p$ such that 
$w' \le s, t$.

Now let $A$ be a maximal antichain of $N \cap \p$, and we will show that 
$f[A]$ is predense below $q$. 
Since $f$ preserves incompatibility, 
it follows that $f[A]$ is a maximal antichain of $\p / q$. 
Let $D$ be the set of $p \in N \cap \p$ for which there is $s \in A$ 
such that $p \le s$. 
Then $D$ is dense in $N \cap \p$. 
Since $q$ is strongly $N$-generic, $D$ is predense below $q$. 
Consider $r \in \p / q$. 
Then there is $u \in D$ and $t$ such that 
$t \le r, u$. 
By the definition of $D$, there is $s \in A$ such that $u \le s$. 
Then $t \le r, s$. 
As $t \le q$, we have that $t = q \wedge t \le q \wedge s = f(s)$. 
Hence $r$ is compatible with $f(s)$, and $f(s) \in f[A]$.

Let $G$ be a $V$-generic filter on $N \cap \p$. 
Let $r \le q$ be given. 
Assume that $r \in (\p / q) / f[G]$, which means that $r$ is compatible with 
every condition in $f[G]$. 
Consider $s \in G$. 
Fix $t \le r, f(s)$. 
Since $f(s) = q \wedge s \le s$, it follows that $t \le r, s$. 
So $r$ and $s$ are compatible. 
Conversely, assume that $r \le q$ and $r$ is compatible 
with every condition in $G$. 
Let $f(s)$ in $f[G]$ be given, where $s \in G$. 
Since $s \in G$, $r$ and $s$ are compatible, so fix $t \le r, s$. 
Then $t \le q$, so $t \le q \wedge s = f(s)$. 
Hence $r$ and $f(s)$ are compatible.
\end{proof}

For a $V$-generic filter $G$ on $N \cap \p$, 
let $(\p / q) / G$ denote the forcing poset consisting of conditions in 
$\p / q$ which are compatible with every condition in $G$. 
By the lemma, 
$$
(\p / q) / G = (\p / q) / f[G].
$$
Since $N \cap \p$ is forcing equivalent to $f[N \cap \p]$, 
it follows that $\p / q$ is forcing equivalent to the two-step iteration 
$$
(N \cap \p) * ((\p / q) / \dot G_{N \cap \p}).
$$
If $H$ is a $V$-generic filter on $\p$ which contains $q$, then 
by Lemma 1.6(1), $H \cap f[N \cap \p]$ 
is a $V$-generic filter on $f[N \cap \p]$. 
Since $f$ is a dense embedding, 
$f^{-1}[H \cap f[N \cap \p]]$ is a $V$-generic filter on $N \cap \p$. 
But this latter set is just equal to $H \cap N$. 
For if $f(p) = q \wedge p$ is in $H$, then since $q \wedge p \le p$, 
$p$ is in $H$. 
And if $p \in H \cap N$, then since $q \in H$, $f(p) = q \wedge p \in H$. 
By Lemma 1.6(1), 
it follows that $H \cap N$ is a $V$-generic filter on $N \cap \p$, 
$H_q := \{ r \in H : r \le q \}$ is a $V[H \cap N]$-generic filter on 
$(\p / q) / (H \cap N)$, and 
$V[H] = V[H \cap N][H_q]$.

\begin{thm}
Let $\p$ be a forcing poset with greatest lower bounds satisfying 
$*(\p,\p)$, $\chi \ge \lambda_\p$ a regular cardinal, and 
$N \prec (H(\chi),\in,\p)$. 
Suppose that $\p$ has universal strongly generic 
conditions on a stationary set, and $q$ is a universal strongly $N$-generic 
condition.

Then the forcing poset $N \cap \p$ forces that the quotient 
$(\p / q) / \dot G_{N \cap \p}$ has universal strongly generic 
conditions on a stationary set. 
Moreover, whenever $H$ is a $V$-generic filter on $\p$ which contains $q$, 
we have that $V[H] = V[H \cap N][K]$, where $H \cap N$ is a $V$-generic filter 
on $N \cap \p$, $K$ is a $V[H \cap N]$-generic filter on 
$(\p / q) / (H \cap N)$, and the pair $(V[H \cap N],V[H])$ 
satisfies the $\omega_1$-approximation property.
\end{thm}

\begin{proof}
We claim that the forcing poset $N \cap \p$ 
forces that the quotient 
$(\p / q) / \dot G_{N \cap \p}$ has universal strongly generic 
conditions on a stationary set. 
In particular, $N \cap \p$ forces that 
$(\p / q) / \dot G_{N \cap \p}$ is strongly proper on a stationary set, and 
hence has the $\omega_1$-approximation property. 

Let $f : N \cap \p \to \p / q$ be the function $f(p) = q \wedge p$. 
By Theorem 2.11, it suffices to show that the regular 
suborder $f[N \cap \p]$ of $\p / q$ satisfies property 
$*(f[N \cap \p],\p / q)$. 
But this follows immediately from the fact that $\p$ satisfies $*(\p,\p)$.

The second conclusion of the theorem follows from this claim together with the 
analysis given prior to the statement of the theorem.
\end{proof}

An example of factoring a forcing poset over an uncountable elementary 
substructure appears in the final argument of Mitchell's theorem on the 
approachability ideal in \cite[Section 3]{mitchell2}. 
Mitchell uses the existence of what he calls tidy strongly generic conditions 
to show that the quotient has the $\omega_1$-approximation property. 
Theorem 4.3 provides a different justification for Mitchell's argument which 
avoids tidy conditions.

\bigskip

For our second variation of Theorem 2.11, we introduce a weakening of strong properness which is useful in situations where we desire quotients to 
have the $\omega_1$-approximation property but 
universal conditions do not 
exist.

\begin{definition}
A forcing poset $\q$ is \emph{non-diagonally strongly proper} if for every 
condition $p$ in $\q$, there are stationarily many 
$N$ in $P_{\omega_1}(H(\lambda_\q))$ such that 
$p \in N$ and there exists $q \le p$ such that $q$ is a 
strongly $N$-generic condition.
\end{definition}

As usual, if we replace $\lambda_\q$ with any cardinal $\chi \ge \lambda_\q$ 
in the definition we get an equivalent statement.

\begin{proposition}
Let $\q$ be a non-diagonally strongly proper forcing poset. 
Then $\q$ satisfies the $\omega_1$-approximation property.
\end{proposition}

The proof is the same as the one sketched after Proposition 2.13.

\begin{thm}
Suppose that $\q$ is a forcing poset with greatest lower bounds satisfying that 
for every condition $p$ in $\q$, there are stationarily many 
$N$ in $P_{\omega_1}(H(\lambda_\q))$ such that 
$p \in N$ and there exists a strongly $N$-generic condition 
extending $p$. 
Let $\p$ be a regular suborder of $\q$ satisfying property $*(\p,\q)$. 
Then $\p$ forces that $\q / \dot G_\p$ is non-diagonally strongly proper. 
In particular, $\p$ forces that $\q / \dot G_\p$ has the 
$\omega_1$-approximation property.
\end{thm}

\begin{proof}
Let $\chi$ be a regular cardinal such that $\q \in H(\chi)$. 
Then $\q$ forces that $\chi$ is regular and 
$\lambda_{\q / \dot G_\p} \le \chi$. 
Let $\dot F$ be a $\p$-name for a function 
$\dot F : (H(\chi)^{V[\dot G_\p]})^{<\omega} \to 
H(\chi)^{V[\dot G_\p]}$, let $s \in \p$, 
and suppose that $s$ forces that $\dot p$ is a condition in $\q / \dot G_\p$. 
We will find an extension of $s$ in $\p$ which forces that there exists a countable 
set $M \subseteq H(\chi)^{V[\dot G_\p]}$ 
which is closed under $\dot F$, contains $\dot p$, and there is a 
strongly $(M,\q / \dot G_\p)$-generic condition below $\dot p$.

Extending $s$ if necessary, assume that for some $p \in \q$, $s$ forces 
that $\dot p = \check p$. 
Since $s$ forces that $\dot p$ is in $\q / \dot G_\p$, $s$ forces that 
$\dot p$ is compatible with every condition in $\dot G_\p$. 
In particular, $s$ forces that $\dot p$ is compatible with $s$. 
So $p$ and $s$ are compatible.

Let $H : H(\chi)^{<\omega} \to H(\chi)$ be a function such that $\p$
forces that $H(\dot a_0,\ldots,\dot a_n) = \dot F(\dot a_0,\ldots,\dot a_n)$ for all 
$\p$-names $\dot a_0, \ldots, \dot a_n$ in $H(\chi)$. 
Since $\q$ is non-diagonally strongly proper, fix a countable set 
$N \prec H(\chi)$ 
closed under $H$ such that $\p$, $\q$, $s$, and $p$ are in $N$ 
and there is $q \le p \wedge s$ which is a strongly $(N,\q)$-generic condition.

As $\p$ is a regular suborder of $\q$, fix $z$ in $\p$ such that 
every extension of $z$ in $\p$ is compatible with $q$. 
Then $z$ forces that $q$ is in $\q / \dot G_\p$. 
Since $z$ is compatible with $q$ in $\q$ 
and $q \le s$, $z$ is compatible with $s$ in $\q$. 
Since $\p$ is a regular suborder and $z$ and $s$ are in $\p$, 
$z$ and $s$ are compatible in $\p$. 
So without loss of generality assume that $z \le s$.

Since $N$ is closed under $H$, $\p$ forces that $N[\dot G_\p]$ is closed 
under $\dot F$. 
By Lemma 2.10, $z$ forces that $q$ is a 
strongly $(N[\dot G_\p],\q / \dot G_\p)$-generic condition. 
Also $z \le s$, and since $p \in N$, 
$z$ forces that $\dot p = \check p$ is in $N[\dot G_\p]$.
\end{proof}

We give an example of a non-diagonally strongly proper forcing poset 
which is not strongly proper on a stationary set.

Consider a stationary set $S \subseteq \omega_1$. 
Recall the forcing poset $\p_S$ for adding a club subset of $S$ 
with finite conditions (\cite[Theorem 3]{AS}). 
A condition in $\p_S$ is a finite set $p$ of ordered pairs 
$\langle \alpha, \beta \rangle$ such that $\alpha \in S$ and 
$\alpha \le \beta < \omega_1$, 
and whenever $\langle \alpha, \beta \rangle$ and $\langle \alpha', \beta' \rangle$ 
are in $p$ then it is not the case that $\alpha < \alpha' \le \beta$. 
Let $q \le p$ in $\p_S$ if $p \subseteq q$.

\begin{proposition}
Let $N$ be a countable elementary substructure of 
$(H(\omega_1),\in,S)$ such that $N \cap \omega_1 \in S$, and 
let $p \in N \cap \p_S$. 
Then $p \cup \{ \langle N \cap \omega_1, N \cap \omega_1 \rangle \}$ 
is a strongly $(N,\p_S)$-generic condition.
\end{proposition}

The proof is straightforward. 
The forcing poset $\p_S$ preserves all cardinals and adds a club 
subset of $S$; see \cite{AS} for the details.

Now consider $S_0$ and $S_1$ which are disjoint stationary subsets 
of $\omega_1$. 
Define $\q$ as the forcing poset consisting of conditions of the 
form $(i,p)$, where $i \in \{ 0, 1 \}$ and $p \in \p_{S_i}$. 
Let $(j,q) \le (i,p)$ if $i = j$ and $q \le p$ in $\p_{S_i}$.

\begin{proposition}
The forcing poset $\q$ is non-diagonally strongly proper.
\end{proposition}

\begin{proof}
Let $(i,p)$ be a condition in $\q$ and  
let $F : H(\omega_1)^{<\omega} \to H(\omega_1)$ be a function. 
Since $S_i$ is stationary, we can fix a countable 
elementary substructure $N$ of $(H(\omega_1),\in,F,S_0,S_1)$ 
such that $N \cap \omega_1 \in S_i$ and $(i,p) \in N$. 
Then Proposition 4.7 implies that 
$(i,p \cup \{ \langle N \cap \omega_1, N \cap \omega_1 \rangle \})$ 
is a strongly $(N,\q)$-generic condition extending $(i,p)$.
\end{proof}

\begin{proposition}
The forcing poset $\q$ is not strongly proper on a stationary set.
\end{proposition}

\begin{proof}
Let $\dot C$ be a $\q$-name for the club subset of $\omega_1$ added by $\q$. 
Suppose that $N$ is an elementary substructure of $H(\chi)$ for some 
regular $\chi > \omega_1$ such that $S_0$, $S_1$, $\q$, and $\dot C$ 
are in $N$. 
Since $S_0$ and $S_1$ are disjoint, $N \cap \omega_1$ cannot be in 
both $S_0$ and $S_1$. 
Fix $j \in \{ 0, 1 \}$ such that $N \cap \omega_1$ is not in $S_j$. 
We claim that the condition $(j,\emptyset)$, which is in $N \cap \q$, 
does not have a strongly $N$-generic extension. 
In fact, it does not have an $N$-generic extension.

Suppose for a contradiction that $(j,q) \le (j,\emptyset)$ is $N$-generic. 
Then $(j,q)$ forces that $N[\dot G] \cap V = N$, and in particular, 
$N[\dot G] \cap \omega_1 = N \cap \omega_1$. 
Since $\dot C \in N$, $\q$ forces that $N[\dot G] \cap \omega_1 \in \dot C$. 
So $(j,q)$ forces that $N \cap \omega_1 \in \dot C$. 
Since $(j,q)$ forces that $\dot C$ is a subset of $S_j$, 
$(j,q)$ forces that $N \cap \omega_1 \in S_j$. 
But $N \cap \omega_1$ is not in $S_j$, so we have a contradiction.
\end{proof}

\section{A Counterexample}

We give an example of a strongly proper 
forcing poset $\q$ and a regular suborder $\p$ of $\q$ 
such that $\p$ forces that 
$\q / \dot G_\p$ is not strongly proper on a stationary set. 
We will make use of the following well known fact.

\begin{lemma}
If $\q$ is strongly proper on a stationary set, then any generic extension 
of $V$ by $\q$ contains a $V$-generic Cohen real.
\end{lemma}

\begin{proof}
(Sketch) Let $G$ be a $V$-generic filter on $\q$. 
By a density argument, there exists a countable 
elementary substructure $N$ 
with $\q \in N$ and a strongly $N$-generic condition $q \in G$. 
Then $N \cap G$ is a $V$-generic filter on $N \cap \q$. 
But $N \cap \q$ is a countable nontrivial forcing poset, and hence is 
forcing equivalent to Cohen forcing $\add(\omega)$.
\end{proof}

So it suffices to define a strongly proper forcing poset $\q$ and a regular 
suborder $\p$ of $\q$ such that $\p$ forces that 
$\q / \dot G_\p$ is nontrivial and does not add reals over $V[\dot G_\p]$. 

Assume that $2^\omega = \omega_1$ and $2^{\omega_1} = \omega_2$. 
Let $\mathcal X$ denote the set of all countable elementary substructures 
of $H(\omega_3)$. 
The forcing poset $\q$ consists of finite coherent adequate subsets of $\mathcal X$, 
ordered by inclusion. 
The definition of coherent adequate is beyond the scope of the paper, although 
we give additional information about adequacy in the next section. 
Roughly speaking, a coherent adequate set of models 
satisfies that any two models in it are membership comparable up to some 
initial segment of the models, and if the models are equal up to some initial 
segment, then they are isomorphic. 
For the complete definition, see \cite[Section 1]{jk25}.

The following lists the properties of the forcing poset $\q$ which we will use. 
For a proof of (1) and (2) see \cite{jk25}. 
For a proof of (3) see \cite{miyamoto}.

\begin{proposition}
The following statements hold:
\begin{enumerate}
\item $\q$ is strongly proper and $\omega_2$-c.c.
\item $\q$ forces CH.
\item $\q$ forces that there exists an $\omega_1$-Kurepa 
tree with $\omega_3$ many distinct branches.
\end{enumerate}
\end{proposition}

Since $\q$ does not have greatest lower bounds, 
we will consider the Boolean completion of $\q$, which we 
denote by $\mathcal B$. 
Then easily $\mathcal B$ has the $\omega_2$-c.c., preserves all cardinals, 
and forces CH. 
Also a straightforward argument shows that since $\q$ is strongly 
proper, so is $\mathcal B$.

Let us find a regular suborder $\p$ of $\mathcal B$ which forces that 
$\mathcal B / \dot G_\p$ is nontrivial and does not add reals over $V[\dot G_\p]$. 
Since $\mathcal B$ forces CH, 
we can fix a sequence $\langle \dot r_i : i < \omega_1 \rangle$ 
of $\mathcal B$-names for subsets of $\omega$ 
such that $\mathcal B$ forces that every subset of $\omega$ is 
equal to $\dot r_i$ for some $i < \omega_1$. 
Moreover, we assume that each name $\dot r_i$ 
is a nice $\mathcal B$-name given by 
a sequence of antichains $\langle A^i_n : n < \omega \rangle$. 
So $(p,\check n)$ is in the name $\dot r_i$ iff $p \in A^i_n$. 
Since $\mathcal B$ is $\omega_2$-c.c., each antichain 
$A^i_n$ has size at most $\omega_1$.

Fix a regular cardinal $\chi > \omega_3$ such that 
$\q$, $\mathcal B$, $\langle \dot r_i : i < \omega_1 \rangle$, and 
$\langle A^i_n : n < \omega \rangle$ for all $i < \omega_1$ 
are members of $H(\chi)$. 
Let $N$ be an elementary substructure of $H(\chi)$ containing these parameters 
such that $N$ has size $\omega_2$ and $N^{\omega_1} \subseteq N$. 
This is possible since we are assuming that $2^{\omega_1} = \omega_2$. 
Now let $\p := N \cap \mathcal B$.

\begin{proposition}
The forcing poset $\p$ is a regular suborder of $\mathcal B$ which forces that 
$\mathcal B / \dot G_\p$ is a nontrivial forcing which does not add reals.
\end{proposition}

\begin{proof}
Suppose that $p$ and $q$ are in $\p$ and are compatible in $\mathcal B$. 
Then $p$ and $q$ are in $N$, so by elementarity, $p \wedge q$ 
is in $N \cap \mathcal B = \p$. 
So they are compatible in $\p$. 
Now let $A$ be a maximal antichain of $\p$. 
Then $A$ is an antichain of $\mathcal B$ and is a subset of $\mathcal B \cap N$. 
Since $\mathcal B$ is $\omega_2$-c.c., $A$ has size at most $\omega_1$. 
But $N^{\omega_1} \subseteq N$, so $A \in N$. 
Suppose for a contradiction that $A$ is not a maximal antichain of $\mathcal B$. 
Then there is a condition in $\mathcal B$ which is incompatible with every member of $A$. 
Since $A \in N$, by elementarity there is a condition in 
$N \cap \mathcal B = \p$ which 
is incompatible with every member of $A$.  
But then $A$ is not maximal in $\p$, which is a contradiction.

The forcing poset $\p$ has size $\omega_2$ and is $\omega_2$-c.c. 
Since $2^{\omega_1} = \omega_2$ in the ground model, $\p$ forces 
that $2^{\omega_1} = \omega_2$ by a standard nice name argument. 
But $\mathcal B$ forces that there is an $\omega_1$-tree with $\omega_3$ 
many distinct branches, and hence $2^{\omega_1} \ge \omega_3$. 
It follows that $\p$ forces that $\mathcal B / \dot G_\p$ is a nontrivial forcing.

To show that $\p$ forces that $\mathcal B / \dot G_\p$ does not add reals, 
let $G$ be a $V$-generic filter on $\mathcal B$. 
Let $G_\p := G \cap N$, which is a $V$-generic filter on $\p$.  
Let $r$ be a subset of $\omega$ in $V[G]$, and we will show that 
$r$ is in $V[G_\p]$. 
By assumption, there is $i < \omega_1$ such that 
$r = r_i$, where $r_i := \dot r_i^G$. 
For each $n < \omega$, $n \in r_i$ iff $A^i_n \cap G \ne \emptyset$. 
But $A^i_n$ is in $N$ and hence is a subset of $N$ since 
$A^i_n$ has size at most $\omega_1$. 
So $n \in r_i$ iff $A^i_n \cap G \cap N \ne \emptyset$ 
iff $A^i_n \cap G_\p \ne \emptyset$. 
It follows that $r_i$, and hence $r$, is in $V[G_\p]$.
\end{proof}

We note that in fact $\p$ forces that $\mathcal B / \dot G_\p$ does not 
have the $\omega_1$-approximation property. 
Recall that $\q$ forces that there exists an $\omega_1$-tree with 
$\omega_3$ many distinct branches. 
By looking at the details of the definition of this tree as presented in 
\cite{miyamoto}, one can argue that this tree already 
exists in a generic extension by $\p$. 
Since $2^{\omega_1} = \omega_2$ in the generic extension by $\p$, 
$\p$ forces that $\q / \dot G_\p$ adds new branches to the tree. 
But any initial segment of such a branch lies in the extension by $\p$, 
and this easily implies that the quotient does not 
have the $\omega_1$-approximation property.

\section{A strongly proper collapse}

We now turn towards proving our consistency result. 
We will construct a model in which $\mathsf{ISP}(\omega_2)$ holds 
and the continuum is greater than $\omega_2$. 
The forcing poset used to obtain this model will be of the form 
$\p \times \add(\omega,\lambda)$, where $\p$ is a strongly proper 
forcing which collapses a supercompact cardinal $\kappa$ to become 
$\omega_2$, and $\add(\omega,\lambda)$ is the forcing which adds 
$\lambda$ many Cohen reals. 

In this section we will describe the forcing poset $\p$. 
Essentially this forcing is the pure side condition forcing 
consisting of finite adequate subsets of $\kappa$ ordered by inclusion, 
where the notion of adequate set is the same as that defined by 
Krueger in \cite{jk21}, except that $\omega_2$ is replaced by a strongly 
inaccessible cardinal $\kappa$.\footnote{Mitchell \cite{mitchell2} was the 
first to define a strongly proper collapse of an 
inaccessible to become $\omega_2$. 
Neeman \cite{neeman} gives another example using his method of 
sequences of models of two types.}

Assume for the rest of the section 
that $\kappa$ is a strongly inaccessible cardinal. 
In particular, $H(\kappa)$ has size $\kappa$. 
Fix a bijection 
$\pi : \kappa \to H(\kappa)$. 
Then the structure $(H(\kappa),\in,\pi)$ has definable Skolem functions. 
For any set $a \subseteq \kappa$, let $Sk(a)$ denote the closure of $a$ 
under some (or equivalently any) set of definable Skolem functions.

Let $C$ be the club set of $\alpha < \kappa$ such that 
$Sk(\alpha) \cap \kappa = \alpha$. 
Let $\Lambda$ be the set of cardinals $\beta < \kappa$ 
such that $\beta$ is a limit 
point of $C$ with uncountable cofinality, and 
for all $\gamma < \beta$, $\gamma^\omega < \beta$. 
Note that $\Lambda$ is stationary in $\kappa$. 
Let $\mathcal X$ be the set of countable subsets $a$ of $\kappa$ such that 
$Sk(a) \cap \kappa = a$ and for all $\gamma \in a$, 
$\sup(C \cap \gamma) \in a$. 

Using $\kappa$ in place of $\omega_2$ and the sets $C$, $\Lambda$, and 
$\mathcal X$ just described, it is possible to develop 
the basic ideas of adequate sets word for word as in \cite{jk21}. 
We will give an overview some of the main points. 
The interested reader is invited to read Sections 1--4 of \cite{jk21} for the 
complete details.

For a set $M \in \mathcal X$, define $\Lambda_M$ as the set of 
$\beta \in \Lambda$ such that 
$$
\beta = \min(\Lambda \setminus \sup(M \cap \beta)).
$$
By Lemma 2.4 of \cite{jk21}, for all $M$ and $N$ in $\mathcal X$, 
$\Lambda_M \cap \Lambda_N$ has a largest element. 
This largest element is defined as $\beta_{M,N}$, 
the \emph{comparison point of $M$ and $N$}. 

One of the most important properties of the comparison point of 
$M$ and $N$ is expressed in the following inclusion:
$$
(M \cup \lim(M)) \cap (N \cup \lim(N)) \subseteq \beta_{M,N}.
$$
This is proved in \cite[Proposition 2.6]{jk21}.

\begin{definition}
A finite set $A \subseteq \mathcal X$ is \emph{adequate} if 
for all $M$ and $N$ in $A$, either 
$M \cap \beta_{M,N} \in Sk(N)$, $N \cap \beta_{M,N} \in Sk(M)$, or 
$M \cap \beta_{M,N} = N \cap \beta_{M,N}$.
\end{definition} 

Note that $A$ is adequate if for all $M$ and $N$ in $A$, 
$\{ M, N \}$ is adequate. 
Since $\omega_1 \le \beta_{M,N}$, if $\{ M, N \}$ is adequate then 
$M \cap \beta_{M,N} \in Sk(N)$ iff $M \cap \omega_1 \in N$, and 
$M \cap \beta_{M,N} = N \cap \beta_{M,N}$ iff 
$M \cap \omega_1 = N \cap \omega_1$.

\begin{proposition}
If $A$ is adequate, $M \in \mathcal X$, and $A \in Sk(M)$, then 
$A \cup \{ M \}$ is adequate.
\end{proposition}

See \cite[Proposition 3.5]{jk21}.

If $N \in \mathcal X$ and $\beta \in \Lambda$, then 
$N \cap \beta$ is in $\mathcal X$ (\cite[Lemma 1.10]{jk21}). 
Let us say that an adequate set $A$ is \emph{$N$-closed} 
if for all $M \in A$, $M \cap \beta_{M,N} \in A$.

\begin{proposition}
Let $A$ be adequate and $N \in A$. 
Then there exists an adequate set $B$ such that 
$A \subseteq B$ and $B$ is $N$-closed.
\end{proposition}

This follows from \cite[Proposition 3.4]{jk21}. 
The next result appears as \cite[Proposition 3.9]{jk21}.

\begin{proposition}
Let $A$ be adequate, $N \in A$, and suppose that $A$ is $N$-closed. 
Let $B$ be adequate such that $A \cap Sk(N) \subseteq B \subseteq Sk(N)$. 
Then $A \cup B$ is adequate.
\end{proposition}

Now we are ready to define the forcing poset $\p$. 
Let $\p$ consist of conditions $A$ such that $A$ is a finite adequate 
subset of $\mathcal X$. 
Let $B \le A$ in $\p$ if $A \subseteq B$.

\begin{lemma}
The forcing poset $\p$ has greatest lower bounds. 
Namely, for compatible conditions $A$ and $B$ in $\p$, 
$A \wedge B$ equals $A \cup B$.
\end{lemma}

\begin{proof}
Suppose that $A$ and $B$ are compatible, and let $C \le A, B$. 
Since the ordering of $\p$ is inclusion, $A \cup B \subseteq C$. 
For all $M$ and $N$ in $A \cup B$, $\{ M, N \}$ is adequate since they are in $C$. 
So $A \cup B$ is adequate. 
Hence $A \cup B \in \p$ and $C \le A \cup B$.
\end{proof}

\begin{lemma}
For all $A$, $B$, and $C$ in $\p$ which are pairwise compatible, 
$A$ is compatible with $B \cup C$.
\end{lemma}

\begin{proof}
Let $M$ and $N$ be in $A \cup (B \cup C)$. 
Then $M$ and $N$ are either both in  
$A \cup B$, $A \cup C$, or $B \cup C$. 
Since each of these three sets is adequate, $\{ M, N \}$ is adequate.
\end{proof}

\begin{proposition}
The forcing poset $\p$ has universal strongly generic conditions on a club.
\end{proposition}

\begin{proof}
Note that $\mathcal X$ is a club subset of $P_{\omega_1}(\kappa)$. 
Clearly $\lambda_\p = \kappa$. 
Consider any set $N$ in $P_{\omega_1}(H(\kappa))$ such that 
$N \cap \kappa \in \mathcal X$. 
We claim that $\{ N \cap \kappa \}$ is a universal strongly $N$-generic condition.

If $A$ is a finite adequate set in $N = Sk(N \cap \kappa)$, then 
by Proposition 6.2, $A \cup \{ N \cap \kappa \}$ is adequate and hence is 
in $\p$. 
Therefore $A$ and $\{ N \cap \kappa \}$ are compatible. 

To show that $\{ N \cap \kappa \}$ is a strongly 
$N$-generic condition, 
let $E$ be the set of $B$ in 
$\p$ such that $\{ N \cap \kappa \} \subseteq B$ 
and $B$ is $N$-closed. 
By Proposition 6.3, $E$ is dense below $\{ N \cap \kappa \}$. 
Define $C \mapsto C \restriction N$ for $C \in E$ by letting 
$C \restriction N := C \cap N$. 
Clearly $C \cap N$ is adequate and hence is in $\p$, 
$C \cap N \in N$, and $C \le C \cap N$. 
If $B$ is a condition in $N$ extending $C \restriction N$, then 
Proposition 6.4 implies that $B \cup C$ is a condition in $\p$ 
below $B$ and $C$, showing that $C$ and $B$ are compatible.
\end{proof}

\begin{proposition}
The forcing poset $\q$ is $\kappa$-c.c.
\end{proposition}

The proof of this proposition is identical to the proof of 
\cite[Proposition 4.4]{jk21}, so we omit it.

\begin{lemma}
The forcing poset $\p$ forces that $\kappa$ equals $\omega_2$.
\end{lemma}

\begin{proof}
Since $\p$ preserves $\omega_1$ and $\kappa$ by the preceding propositions, 
it suffices to show that for any regular 
cardinal $\mu$ such that $\omega_1 < \mu < \kappa$, 
$\p$ collapses $\mu$ to have size $\omega_1$. 
So let such a $\mu$ be given and let $G$ be a $V$-generic filter on $\p$.

Given any condition $A$ and any ordinal $\gamma < \kappa$, we can 
fix $M$ in $\mathcal X$ such that $A$, $\gamma$, and $\mu$ are in $Sk(M)$. 
Then $A \cup \{ M \}$ is a condition in $\p$ by Proposition 6.2. 
It follows by a density argument that the set 
$$
F = \{ M : \exists A \in G \ ( M \in A \ \land \ \mu \in M) \}
$$
has union equal to $\kappa$. 
In particular, $\bigcup \{ M \cap \mu : M \in F \} = \mu$.

Now for any $M$ and $N$ in $F$, $\mu \in M \cap N$ implies that 
$\mu < \beta_{M,N}$. 
Fix $A$ in $G$ such that $M$ and $N$ are in $A$. 
Since $A$ is adequate and $\mu < \beta_{M,N}$, 
$M \cap \mu$ and $N \cap \mu$ are either equal or one is a 
proper subset of the other. 
Moreover, since $M \cap \beta_{M,N} \subseteq N$ iff $M \cap \omega_1 \le 
N \cap \omega_1$, 
it follows that the set $\{ M \cap \mu : M \in F \}$ is a chain well-ordered 
by inclusion. 
As each set in this chain is countable, 
it must have length at most $\omega_1$. 
So $\mu$ is the union of $\omega_1$ many countable 
sets, which implies that $\mu$ is collapsed to have size $\omega_1$ in $V[G]$.
\end{proof}

This completes our treatment of a strongly proper collapse. 
We note that some variations are possible. 
By \cite{jk24}, it is possible to eliminate $N$-closure which is used in 
Propositions 6.4 and 6.7, so that the family of models $\mathcal X$ can 
be chosen not to be closed under intersections. 
An alternative development of these ideas can be made in the 
context of coherent adequate set forcing as described in 
\cite{jk23} and \cite{jk25}, to produce a strongly proper collapse 
which preserves CH.

We also point out that if $\kappa$ is a Mahlo cardinal, then the forcing poset 
$\p$ above forces that there are no special Aronszajn trees on $\omega_2$, 
and if $\kappa$ is weakly compact then $\p$ forces that there are no 
Aronszajn trees on $\omega_2$. 
These facts follow by arguments similar to 
Mitchell's original proof in \cite{mitchelltree} 
using Corollary 2.14. 
Neeman obtained similar results in \cite[Section 5]{neeman}.

\section{The consistency result}

We start with a model in which GCH holds, $\kappa$ is a 
supercompact cardinal, and $\lambda$ is a cardinal 
of uncountable cofinality with $\kappa \le \lambda$. 
We will define a forcing poset $\p \times \q$ which collapses 
$\kappa$ to become $\omega_2$, forces $2^\omega = \lambda$, and 
forces that $\mathsf{ISP}(\omega_2)$ holds.

Let $\p$ be the forcing poset described in the preceding section consisting of 
finite adequate collections of countable subsets of $\kappa$, ordered by inclusion. 
Let $\q$ be $\add(\omega,\lambda)$. 
Conditions in $\q$ are finite partial functions 
from $\omega \times \lambda$ into $2$, ordered by inclusion.

The forcing poset $\q$ is $\kappa$-Knaster. 
In other words, if $\{ q_i : i < \kappa \}$ is a subset of $\q$, 
then there is a cofinal set $X \subseteq \kappa$ such that 
for all $i < j$ in $X$, $q_i$ and $q_j$ are compatible. 
This follows from the $\Delta$-system lemma by a standard argument. 
Since $\p$ is $\kappa$-c.c., it follows easily that 
$\p \times \q$ is $\kappa$-c.c. 
Also $\p \times \q$ has size $\lambda$. 
It follows by a standard argument using nice names that 
$\p \times \q$ forces that $2^\omega = \lambda$.

\begin{lemma}
If $p$ and $q$ are in $\q$, then $p \cup q$ is the greatest lower 
bound of $p$ and $q$. 
If $p$, $q$, and $r$ 
are pairwise compatible conditions in $\q$, then 
$p$ is compatible with $q \wedge r$.
\end{lemma}

The proof is easy.

\begin{lemma}
The forcing poset $\q$ has universal strongly 
generic conditions on a club.
\end{lemma}

\begin{proof}
Let $N$ be a countable elementary substructure of $H(\lambda)$. 
We claim that the empty condition is a 
universal strongly $N$-generic condition. 
Clearly it is compatible with every condition in $N \cap \q$.

For each $r \in \q$, let $r \restriction N := r \cap N$. 
Then $r \restriction N \in N \cap \q$ and $r \le r \restriction N$. 
Suppose that $v \le r \restriction N$ is in $N \cap \q$. 
We claim that $r$ and $v$ are compatible. 
If $(i,n) \in \dom(r) \cap \dom(v)$, then $(i,n) \in \dom(r) \cap N$. 
Also $r(i,n) \in N$, so $(i,n,r(i,n)) \in r \cap N$. 
Since $v$ extends $r \cap N$, $r(i,n) = v(i,n)$. 
It follows that $r \cup v$ is a condition below $r$ and $v$.
\end{proof}

\begin{proposition}
The forcing poset $\p \times \q$ satisfies:
\begin{enumerate}
\item $\p \times \q$ has greatest lower bounds;
\item for all pairwise compatible conditions $x$, $y$, and $z$ 
in $\p \times \q$, $x$ is compatible with $y \wedge z$;
\item $\p \times \q$ has universal strongly generic conditions 
on a club;
\item $\p \times \q$ is strongly proper.
\end{enumerate}
\end{proposition}

\begin{proof}
(1) and (2) follow from Lemmas 3.1, 6.5, 6.6, and 7.1. 
(3) follows from Lemma 3.2, Proposition 6.7, and Lemma 7.2.
\end{proof}

Since $\p \times \q$ is strongly proper it preserves $\omega_1$, and 
since it is $\kappa$-c.c., it preserves all cardinals greater than or 
equal to $\kappa$. 
By Lemma 6.9, every cardinal $\mu$ with $\omega_1 < \mu < \kappa$ 
is collapsed to have size $\omega_1$. 
It follows that $\p \times \q$ forces that $\kappa$ is equal to $\omega_2$. 
As noted above, $\p \times \q$ forces that $2^\omega = \lambda$. 

\bigskip

It remains to show that $\p \times \q$ forces that 
$\mathsf{ISP}(\omega_2)$ holds. 
So fix a regular cardinal $\chi$ such that $\p \times \q \in H(\chi)$. 
Let $\dot F$ be a $(\p \times \q)$-name for a function 
$\dot F : (H(\chi)^{V[\dot G_{\p \times \q}]})^{<\omega} \to 
H(\chi)^{V[\dot G_{\p \times \q}]}$. 
We will prove that $\p \times \q$ forces that there exists a set $N$ satisfying:
\begin{enumerate}
\item $N$ is in $P_{\kappa}(H(\chi))$;
\item $N \cap \kappa \in \kappa$;
\item $N \prec H(\chi)$;
\item $N$ is closed under $\dot F$;
\item $N$ is an $\omega_1$-guessing model.
\end{enumerate}

Since $\kappa$ is supercompact, we can fix an elementary embedding 
$j : V \to M$ with critical point $\kappa$ such that 
$j(\kappa) > |H(\chi)|$ and $M^{|H(\chi)|} \subseteq M$.

\begin{lemma}
In $V$, the function $j \restriction \p \times \q$ is a regular embedding of 
$\p \times \q$ into $j(\p \times \q)$. 
In particular, 
the forcing poset $j[\p \times \q]$ is a regular suborder of $j(\p \times \q)$.
\end{lemma}

\begin{proof}
Properties (1) and (2) of Definition 4.1 follow immediately from the fact 
that $j$ is an elementary embedding. 
For (3), let $A$ be a maximal antichain of $\p \times \q$. 
Since $\p \times \q$ is $\kappa$-c.c., $|A| < \kappa$. 
Therefore $j(A) = j[A]$. 
By elementarity, in $M$ the set $j(A)$ is a maximal antichain of $j(\p \times \q)$. 
But being a maximal antichain is upwards absolute, so $j(A) = j[A]$ is 
a maximal antichain of $j(\p \times \q)$.
\end{proof}

Since being a regular suborder is downwards absolute, $j[\p \times \q]$ is a 
regular suborder of $j(\p \times \q)$ in the model $M$.

Let $G \times H$ be a $V$-generic filter on $\p \times \q$. 
Let $F := \dot F^{G \times H}$. 
We will prove that in $V[G \times H]$, there exists a set $N$ satisfying 
properties (1)--(5) above. 

Since $j \restriction \p \times \q$ is an isomorphism of $\p \times \q$ 
onto $j[\p \times \q]$, $j[G \times H]$ is a $V$-generic filter on 
$j[\p \times \q]$ and $V[G \times H] = V[j[G \times H]]$. 
Let $K$ be a $V[G \times H]$-generic filter on the quotient forcing 
$j(\p \times \q) / j[G \times H]$. 
By Lemma 1.6(2), $K$ is a $V$-generic filter on $j(\p \times \q)$ and 
$V[G \times H][K] = V[K]$.
Also by Lemma 1.6(2), $j[G \times H] = K \cap j[\p \times \q]$. 
In particular, $j[G \times H] \subseteq K$. 
Hence we can extend the elementary embedding $j$ in 
$V[K]$ to $j : V[G \times H] \to M[K]$ 
by letting $j(\dot a^{G \times H}) = j(\dot a)^K$.

By the elementarity of $j$, it suffices to prove that $M[K]$ models that there 
exists a set $N$ satisfying:
\begin{enumerate}
\item $N$ is in $P_{j(\kappa)}(H(j(\chi)))$;
\item $N \cap j(\kappa) \in j(\kappa)$;
\item $N \prec H(j(\chi))$;
\item $N$ is closed under $j(F)$;
\item $N$ is an $\omega_1$-guessing model.
\end{enumerate}

\bigskip

Let $N := j[H(\chi)^{V[G \times H]}]$.

\bigskip

First, we prove that $N$ is in $M[K]$. 
Since $\p \times \q \in H(\chi)^V$, 
$H(\chi)^{V[G \times H]} = H(\chi)^V[G \times H]$. 
As $M^{|H(\theta)^V|} \subseteq M$, 
$H(\theta)^V$ and $j \restriction H(\theta)^V$ are in $M$. 
By the definition of the extended embedding $j$, 
$N = \{ j({\dot a}^{G \times H}) : \dot a \in H(\chi)^V \} = 
\{ j(\dot a)^K : \dot a \in H(\chi)^V \}$. 
Since $H(\chi)^V$, $j \restriction H(\chi)^V$, and $K$ are in $M[K]$, 
so is $N$.

\bigskip

Now let us check that $N$ satisfies properties (1)--(5).

\bigskip

(1) Obviously $N \subseteq H(j(\chi))$ in $M[K]$. 
Let $f : |H(\theta)^V| \to H(\theta)^V$ be a bijection in $V$. 
The surjection $g : |H(\theta)^V| \to N$ given by 
$g(\alpha) = j(f(\alpha))^K$ is in $M[K]$. 
Since $j(\kappa) > |H(\theta)^V|$, 
in $M[K]$ we have that $N \in P_{j(\kappa)}(H(j(\chi)))$. 

\bigskip

(2) As $\kappa$ is the critical point of $j$, $N \cap j(\kappa)$ equals $\kappa$, 
which is in $j(\kappa)$.

\bigskip

(3) Let $L : (H(\chi)^{V[G \times H]})^{<\omega} \to H(\chi)^{V[G \times H]}$ 
be a Skolem function for the structure $(H(\chi),\in)$. 
By elementarity, $j(L)$ is a Skolem function in $M[K]$ for the structure 
$(H(j(\chi)),\in)$. 
Easily $N$ is closed under $j(L)$. 
So $N \prec H(j(\chi))$ in $M[K]$.

\bigskip

(4) Since $H(\chi)^{V[G \times H]}$ is obviously closed under $F$, 
easily $N$ is closed under $j(F)$.

\bigskip

(5) To show that $N$ is an $\omega_1$-guessing model in $M[K]$, 
by Lemma 1.10 it suffices to show that the pair 
$(\overline N,M[K])$ satisfies the $\omega_1$-approximation property, 
where $\overline N$ is the transitive collapse of $N$. 
Since $N = j[H(\chi)^{V[G \times H]}]$ is isomorphic to 
$H(\chi)^{V[G \times H]}$, which is transitive, 
it follows that $\overline N = H(\chi)^{V[G \times H]}$. 
As $H(\chi)^V = H(\chi)^M$, 
$H(\chi)^{V[G \times H]} = H(\chi)^{M[G \times H]}$. 
So it suffices to show that the pair 
$(H(\chi)^{M[G \times H]},M[K])$ satisfies the 
$\omega_1$-approximation property.

By Proposition 7.3, the forcing poset $\p \times \q$ has greatest lower bounds and 
has universal strongly generic conditions on a stationary set. 
And for all pairwise compatible conditions $x$, $y$, and $z$ in 
$\p \times \q$, $x$ is compatible with $y \wedge z$. 
By elementarity, the same properties are satisfied by $j(\p \times \q)$ 
in the model $M$. 
In particular, by Lemma 7.4 
$j[\p \times \q]$ is a regular suborder of $j(\p \times \q)$ 
satisfying property $*(j[\p \times \q],j(\p \times \q))$. 
By Corollary 2.14, the forcing poset $j(\p \times \q) / j[G \times H]$ 
satisfies the $\omega_1$-approximation property in the model $M[G \times H]$. 
Hence the pair $(M[G \times H],M[K])$ 
satisfies the $\omega_1$-approximation property. 

In $M[K]$ let $d$ be a bounded subset of $H(\chi)^{M[G \times H]} \cap On$ 
such that for any countable set $b \in H(\chi)^{M[G \times H]}$, 
$b \cap d \in H(\chi)^{M[G \times H]}$. 
Since $H(\chi)^{M[G \times H]}$ contains all of its countable subsets in 
$M[G \times H]$ by the regularity of $\chi$, the $\omega_1$-approximation 
property of $(M[G \times H],M[K])$ implies that $d \in M[G \times H]$. 
But $H(\chi)^{M[G \times H]}$ is closed under bounded sets of ordinals, 
so $d \in H(\chi)^{M[G \times H]}$.

\bibliographystyle{plain}
\bibliography{paper26}

\begin{thebibliography}{10}

\bibitem{AS}
U.~Abraham and S.~Shelah.
\newblock Forcing closed unbounded sets.
\newblock {\em J. Symbolic Logic}, 48(3):643--657, 1983.

\bibitem{friedman2}
S.D. Friedman.
\newblock {BPFA} and inner models.
\newblock {\em Annals of the Japan Association for Philosophy of Science},
  19:29--36, 2011.

\bibitem{jk21}
J.~Krueger.
\newblock Forcing with adequate sets of models as side conditions.
\newblock Submitted.

\bibitem{jk23}
J.~Krueger.
\newblock Coherent adequate sets and forcing square.
\newblock {\em Fund. Math.}, 224:279--300, 2014.

\bibitem{jk24}
J.~Krueger.
\newblock Adding a club with finite conditions, part {II}.
\newblock {\em Arch. Math. Logic}, 54(1-2):161--172, 2015.

\bibitem{jk25}
J.~Krueger and M.A. Mota.
\newblock Coherent adequate forcing and preserving {CH}.
\newblock Submitted.

\bibitem{mitchelltree}
W.~Mitchell.
\newblock Aronszajn trees and the independence of the transfer property.
\newblock {\em Ann. Math. Logic}, pages 21--46, 1972.

\bibitem{mitchell1}
W.~Mitchell.
\newblock On the {H}amkins approximation property.
\newblock {\em Ann. Pure Appl. Logic}, 144(1-3):126--129, 2006.

\bibitem{mitchell2}
W.~Mitchell.
\newblock {$I[\omega_2]$} can be the nonstationary ideal on {${\rm
  Cof}(\omega_1)$}.
\newblock {\em Trans. Amer. Math. Soc.}, 361(2):561–601, 2009.

\bibitem{miyamoto}
T.~Miyamoto.
\newblock Forcing a quagmire via matrices of models.
\newblock Preprint.

\bibitem{neeman}
I.~Neeman.
\newblock Forcing with sequences of models of two types.
\newblock {\em Notre Dame J. Form. Log.}, 55(2):265--298, 2014.

\bibitem{shelah}
S.~Shelah.
\newblock {\em Proper and Improper Forcing}.
\newblock Perspectives in Mathematical Logic. Springer-Verlag, Berlin, second
  edition, 1998.

\bibitem{weiss1}
M.~Viale and C.~Weiss.
\newblock On the consistency strength of the proper forcing axiom.
\newblock {\em Adv. Math.}, 228(5):2672--2687, 2011.

\bibitem{weiss2}
C.~Weiss.
\newblock The combinatorial essence of supercompactness.
\newblock {\em Ann. Pure Appl. Logic}, 163(11):1710--1717, 2012.

\end{thebibliography}

\end{document}